\documentclass[12pt]{amsart}
\usepackage{amssymb,amscd}
\usepackage{verbatim}

\usepackage{amsmath,amssymb,graphicx,mathrsfs}   
\usepackage[colorlinks=true,allcolors = blue]{hyperref} 

\usepackage{tikz}
\usetikzlibrary{matrix}

\usepackage[all]{xy}

\let\frak\mathfrak
\let\Bbb\mathbb

\def\>{\relax\ifmmode\mskip.666667\thinmuskip\relax\else\kern.111111em\fi}
\def\<{\relax\ifmmode\mskip-.333333\thinmuskip\relax\else\kern-.0555556em\fi}
\def\vsk#1>{\vskip#1\baselineskip}
\def\vv#1>{\vadjust{\vsk#1>}\ignorespaces}
\def\vvn#1>{\vadjust{\nobreak\vsk#1>\nobreak}\ignorespaces}

  \let\ssize\scriptstyle
\let\sssize\scriptscriptstyle

\let\Medskip\medskip
\def\medskip{\par\Medskip}
\let\Bigskip\bigskip
\def\bigskip{\par\Bigskip}

\let\Maketitle\maketitle
\def\maketitle{\Maketitle\thispagestyle{empty}\let\maketitle\empty}

\newtheorem{thm}{Theorem}[section]
\newtheorem{cor}[thm]{Corollary}
\newtheorem{lem}[thm]{Lemma}

\theoremstyle{definition}                                  
\numberwithin{equation}{section}

\theoremstyle{definition}
\newtheorem*{rem}{Remark}
\newtheorem*{example}{Example}

\let\mc\mathcal
\let\nc\newcommand

\let\al\alpha

\let\la\lambda

\let\phi\varphi
\let\si\sigma

\let\om\omega

\let\der\partial

\let\ox\otimes

\let\geq\geqslant

\let\leq\leqslant

\let\on\operatorname
\let\bi\bibitem
\let\bs\boldsymbol

\def\C{{\mathbb C}}
\def\Z{{\mathbb Z}}

\def\F{{\mathbb F}}   

\def\+#1{^{\{#1\}}}

\def\beq{\begin{equation}}
\def\eeq{\end{equation}}
\def\be{\begin{equation*}}
\def\ee{\end{equation*}}

\nc{\bea}{\begin{eqnarray*}}
\nc{\eea}{\end{eqnarray*}}
\nc{\bean}{\begin{eqnarray}}
\nc{\eean}{\end{eqnarray}}

\def\n{{\mathfrak n}}

\let\ga\gamma
\let\Ga\Gamma

\nc{\Il}{{\mc I_{\bs\la}}}
\nc{\bla}{{\bs\la}}
\nc{\Fla}{\F_\bla}
\nc{\tfl}{{T^*\Fla}}
\nc{\GL}{{GL_n(\C)}}
\nc{\GLC}{{GL_n(\C)\times\C^*}}

\let\sd s 

\def\ddk_#1{\kk_{#1}\<\>\frac\der{\der\<\>\kk_{#1}}}

\def\bul{\mathbin{\raise.2ex\hbox{$\sssize\bullet$}}}
\def\intt{\mathchoice
{\mathop{\raise.2ex\rlap{$\,\,\ssize\backslash$}{\intop}}\nolimits}
{\mathop{\raise.3ex\rlap{$\,\sssize\backslash$}{\intop}}\nolimits}
{\mathop{\raise.1ex\rlap{$\sssize\>\backslash$}{\intop}}\nolimits}
{\mathop{\rlap{$\sssize\<\>\backslash$}{\intop}}\nolimits}}

\let\kk q 
\let\cc c

\let\Ko K

\def\GZ/{Gelfand-Zetlin}
\def\KZ/{{\slshape KZ\/}}
\def\qKZ/{{\slshape qKZ\/}}
\def\XXX/{{\slshape XXX\/}}

\def\Sym{\on{Sym}}

\nc{\A}{{\mc C}}

\def\Sing{{\on{Sing}}}

\def\slt{{\frak{sl}_2}}

\nc{\hsl}{\widehat{{\frak{sl}_2}}}

\nc{\BC}{{ \mathbb C}}
\nc{\lra}{\longrightarrow}
\nc{\CO}{{\mathcal{O}}}
\nc{\BZ}{{ \mathbb Z}}
\nc{\hfn}{\hat{\frak{n}}}

\def\slth{{\frak{sl}_3}}

\begin{document}

\hrule width0pt
\vsk->

\title[The $\F_p$-Selberg integral of type $A_n$]
{The $\F_p$-Selberg integral of type $A_n$}

\author[R.\,Rim\'anyi and A.\:Varchenko]
{Rich\'ard Rim\'anyi$^{\diamond}$ and  Alexander Varchenko$^{\star}$}

\maketitle

\begin{center}
{\it $^{\diamond, \star}$ Department of Mathematics, University
of North Carolina at Chapel Hill\\ Chapel Hill, NC 27599-3250, USA\/}

\vsk.5>
{\it $^{ \star}$ Faculty of Mathematics and Mechanics, Lomonosov Moscow State
University\\ Leninskiye Gory 1, 119991 Moscow GSP-1, Russia\/}

\vsk.5>
 {\it $^{ \star}$ Moscow Center of Fundamental and Applied Mathematics
\\ Leninskiye Gory 1, 119991 Moscow GSP-1, Russia\/}

\end{center}

\vsk>
{\leftskip3pc \rightskip\leftskip \parindent0pt \Small
{\it Key words\/}:  KZ and dynamical equations,  reduction modulo $p$, Selberg integrals,
$\F_p$-integrals

\vsk.6>
{\it 2010 Mathematics Subject Classification\/}: 13A35 (33C60, 32G20) 
\par}

{\let\thefootnote\relax
\footnotetext{\vsk-.8>\noindent
$^\diamond\<${\sl E\>-mail}:\enspace  rimanyi@email.unc.edu, 
supported in part by Simons Foundation grant 523882
\\
$^\star\<${\sl E\>-mail}:\enspace anv@email.unc.edu,
supported in part by NSF grant DMS-1954266}}

\begin{abstract}

We  prove an $\F_p$-Selberg integral formula of type $A_n$, in which
the $\F_p$-Selberg integral is an element of the finite field $\F_p$ with odd prime number $p$ of elements.
The formula is motivated by analogy between  multidimensional hypergeometric solutions of the
KZ equations and
 polynomial solutions of the same equations reduced modulo $p$. For the type $A_1$ the formula 
 was  proved in a previous paper by the authors.

\end{abstract}

{\small\tableofcontents\par}

\setcounter{footnote}{0}
\renewcommand{\thefootnote}{\arabic{footnote}}

\section{Introduction}

In 1944 Atle Selberg proved the following integral formula:
\bean
\label{clS}
&&
\int_0^1\dots\int_0^1 \prod_{1\leq i<j\leq k} (x_i-x_j)^{2\ga} \prod_{i=1}^kx_i^{\al-1} (1-x_i)^{\beta-1}\ dx_1\dots dx_k
\\
&&
\notag
\phantom{aaaaaa}
=\
\prod_{j=1}^k \frac{\Ga(1+j\ga)}{\Ga(1+\ga)}\,
\frac{\Ga(\al+(j-1)\ga)\,\Ga(\beta+(j-1)\ga)}
{\Ga(\al+\beta + (k+j-2)\ga)}\,,
\eean
see  \cite{Se, AAR}. 
Hundreds of papers are devoted to the generalizations of the Selberg
integral formula and its applications, see for example \cite{AAR, FW} and references 
therein.  There are $q$-analysis versions of the formula,
the generalizations associated with Lie algebras, elliptic versions, finite field versions,
see some references  in \cite{AAR, FW, As, Ha, Ka, Op, Ch, TV1, TV2, TV4, Wa1, Wa2, Sp, R, FSV, An, Ev}.
In the finite field versions,
 one considers additive and multiplicative characters of a finite field, which 
map the  field to the field of
complex numbers, and forms an analog of equation \eqref{clS}, in which  both sides are complex numbers.
The simplest of such formulas is the classical relation between Jacobi and Gauss 
sums, see \cite{AAR, An, Ev}.

\vsk.2>
In  \cite{RV} 
we suggested another  version of the Selberg integral formula, in which the
$\F_p$-Selberg integral is an element of the  finite field $\F_p$ with an odd prime number $p$  of elements.

\vsk.2>
Our motivation in \cite{RV} 
came from the theory of Knizhnik-Zamolod\-chikov 
(KZ) equations, see \cite{KZ, EFK}. These are  the
  systems of linear differential equations, satisfied by
conformal blocks on the sphere in the WZW model of conformal field theory.  
The KZ equations were solved in multidimensional hypergeometric integrals in \cite{SV1}, see also \cite{V1,V2}.
The following general principle was formulated in \cite{MuV}: \ if an example of the KZ-type
 equations has a one-dimensional space of solutions, then the 
corresponding multidimensional hypergeometric integral can be evaluated explicitly. 
As an illustration of that principle in \cite{MuV},  an example of  the $\slt$  differential KZ 
equations with a
one-dimensional space of solutions 
was considered, the corresponding multidimensional hypergeometric integral was reduced to the 
Selberg integral and  then evaluated by formula \eqref{clS}. See other illustrations in \cite{FV, FSV, TV1, TV2, TV4, V3, RTVZ}.

\vsk.2>

Recently in \cite{SV2} the KZ equations were considered modulo a prime number $p$ and polynomial solutions of the 
reduced equations were constructed, see also \cite{SlV, V3, V4, V5, V6, V7}. The construction
 is analogous to the construction
of the multidimensional hypergeometric solutions,
and the constructed polynomial solutions were called 
the  $\F_p$-hypergeometric solutions.

\vsk.2>  

In  \cite{RV} we considered  the reduction modulo $p$ of the same example of the $\slt$ 
differential KZ equations, 
that led in \cite{MuV} to the  Selberg integral. We evaluated the corrersponding $\F_p$-hypergeometric solution 
by analogy with the evaluation of the Selberg integral and
obtained the $\F_p$-Selberg integral formula in  \cite[Theorem 4.1]{RV}.

\vsk.2>
In \cite{TV4} the Selberg integral formula of type $A_2$ was  proved,
\bean
\label{f2}
&&
\int_{C_c^{k_1,k_2}[0,1]}
 \prod_{i=1}^{k_1} t_i^{\al_1-1}(1-t_i)^{\beta_1-1}\, \prod_{j=1}^{k_2}(1-s_j)^{\beta_2-1} 
\,\prod_{i=1}^{k_1}\prod_{j=1}^{k_2}  |s_j-t_i|^{-\ga}
\\
\notag
&&
\phantom{aaaaaaa}
\times\,
\prod_{1\leq i<i'\leq k_1} (t_i-t_{i'})^{2\ga}
\prod_{1\leq j<j'\leq k_2} (s_j-s_{j'})^{2\ga}\, dt_1\dots dt_{k_1}\,ds_1\dots ds_{k_2} 
\\
\notag
&&
= \prod_{i=1}^{k_1-k_2}
\frac{\Ga(\beta_1+(i-1)\ga)}
{\Ga(\al+\beta_1+(i+k_1-2)\ga)}
\\
\notag
 &&
\times\,
\prod_{i=1}^{k_2}\
\frac{\Ga(\beta_2+(i-1)\ga)}
{\Ga(1+\beta_2+(i+k_2-k_1-2)\ga)}\,
\
\frac{\Ga(\beta_1+\beta_2+(i-2)\ga)}
{\Ga(\al+\beta_1+\beta_2+(i+k_2-3)\ga)}
\\
&&
\notag
\times \
\prod_{i=1}^{k_2}
\frac{\Ga(1+(i-k_1-1)\ga)\,\Ga(i\ga)}
{\Ga(\ga)}\
\prod_{i=1}^{k_1}
\frac{\Ga(\al+(i-1)\ga)\,\Ga(i\ga)}
{\Ga(\ga)}\,.
\eean
The starting point of this formula
was an example of the joint system of the $\slth$ trigonometric differential KZ  equations and associated dynamical 
difference equations, an example in which the space of solutions is one-dimensional.
The $A_n$-type Selberg integral formula for aribitrary $n$ was obtained in \cite{Wa1, Wa2}, see also \cite{FW}. 

\vsk.2>
In this paper we consider the reduction modulo $p$ of the same
example of the joint system 
of the $\frak{sl}_{n+1}$ trigonometric differential KZ  equations and associated dynamical 
difference equations, which led in \cite{TV4, Wa1} to the $A_n$-type Selberg integral formula.
Using the reduction modulo $p$ of these differential and difference equations we obtain our
$A_n$-type $\F_p$-Selberg integral formula in Theorem \ref{thm main n}.

\smallskip

The paper is organized as follows. In Section \ref{sec 2} we collect useful facts.
In Section \ref{sec 3} we introduce the notion of $\F_p$-integral and discuss the
integral formula for the $\F_p$-beta integral.  In Section \ref{sec 3} we formulate our main  Theorem
\ref{thm main n}, containing the $A_n$-type 
$\F_p$-Selberg integral formula. For $n=1$ it is \cite[Theorem 4.1]{RV}.
Theorem \ref{thm main n} is proved by induction on
$n$ in the remaining part of the paper. 
 In Section \ref{sec 3} we also prove Theorem \ref{thm ind}, which is used in
the transition from the $A_{n-1}$-type formula to the $A_n$-type formula.  
In Section \ref{sec 4} we sketch the proof of formula \eqref{f2}   following  \cite{TV4}.
In Section \ref{sec proofs} we adapt this proof to prove the $A_2$-type $\F_p$-Selberg integral formula.
 The proof of the $A_n$-type $\F_p$-Selberg integral formula for $n>2$ is similar, see Section \ref{sec sketch}.

\smallskip
The authors thank I.\,Cherednik, P.\,Etingof, E.\,Rains for useful discussions.

\section{Preliminary remarks}
\label{sec 2}

Let $p$ be an odd prime number throughout the paper.

\subsection{Cancellation of factorials}

\begin{lem}
\label{lem ca}

If $a,b$ are nonnegative integers and  $a+b=p-1$, then in $\F_p$ we have
\bean
\label{ca}
a!\,b!\,=\,(-1)^{a+1}\,.
\eean

\end{lem}

\begin{proof} We have $a!=(-1)^a(p-1)\dots(p-a)$ and $p-a= b+1$. Hence
$a!\,b! = (-1)^a(p-1)!= (-1)^{a+1}$ by Wilson's Theorem.
\end{proof}

\subsection{Dyson's formula}

We shall use Dyson's formula
\bean
\label{DF}
\on{C.T.} \prod_{1\leq i<j\leq k}(1-x_i/x_j)^c (1-x_j/x_i)^c = \frac{(kc)!}{(c!)^k}\,,
\eean
where C.T. denotes the constant term. See the formula in \cite[Section 8.8]{AAR}.

\subsection{$\F_p$-Integrals}
\label{sec 3}

Let  $M$  an $\F_p$-module. Let $P(x_1,\dots,x_k)$ be a polynomial
with coefficients in $M$,
\bean
\label{St}
P(x_1,\dots,x_k) = \sum_{d}\, c_d \,x_1^{d_1}\dots x_k^{d_k}.
\eean
Let $l=(l_1,\dots,l_k)\in \Z_{>0}^k$. The coefficient
$c_{l_1p-1,\dots,l_kp-1}$ is called the {\it $\F_p$-integral over the $p$-cycle $[l_1,\dots,l_k]_p$}
and is denoted by $\int_{[l_1,\dots,l_k]_p} P(x_1,\dots,x_k)\,dx_1\dots dx_k$.

\begin{lem}
\label{lem sym}

For $i=1,\dots,k-1$ we have
\bean
\label{sym}
&&
\int_{[l_1,\dots,l_{i+1}, l_i,\dots, l_k]_p} P(x_1,\dots,x_{i+1},x_i,\dots,x_k) dx_1\dots dx_k
\\
\notag
&&
\phantom{aaaaa}
=\ 
\int_{[l_1,\dots,l_k]_p} P(x_1,\dots,x_k)\,dx_1\dots dx_k\,.
\eean
\qed
\end{lem}

\begin{lem}
\label{lem St}
For any $i=1,\dots,k$, we have
\bea
\int_{[l_1,\dots,l_k]_p} \frac{\der P}{\der x_i}(x_1,\dots,x_k)\,dx_1\dots dx_k = 0\,.
\eea
\qed
\end{lem}

Let $\bs k=(k_1,\dots,k_n) \in \Z^n_{> 0}$ and 
\bean
\label{def  bsk}
[\bs k]_p := [(1)_{k_1};(k_1)_{k_2};\dots;
(k_{n-1})_{k_{n}}]_p,
\eean
where $x_y$ denotes the $y$-tuple  $(x,\dots,x)$.
For example for $n=2$, $\bs k=(3,2)$, we have $[\bs k]_p = [1,1,1;3,3]_p$.

\subsection{$\F_p$-Beta integral}
For nonnegative integers the classical beta integral formula says
\bean
\label{bk}
\int_{0}^{1} x^{a}(1-x)^{b} dx =\frac{a!\,b!}{(a+b+1)!}\,.
\eean

\begin{thm} [\cite{V7}]
\label{thm bfun}
Let
$a<p$, $b<p$,  $p-1\leq a+b$. 
Then in $\F_p$ we have
\bean
\label{bf1}
\int_{[1]_p} x^a(1-x)^b dx\,
=  \, - \,\frac{a!\,b!}{(a+b+1-p)!}\,.
\eean
If $a+b<p-1$, then 
\bean
\label{bf2}
\int_{[1]_p} x^a(1-x)^b dx\,
=  0\,.
\eean

\end{thm}

\section{ $\F_p$-Selberg integral of type $A_n$} 
\label{sec 3}

\subsection{Admissible parameters}
\label{sec adp}

Let $\bs k=(k_1,\dots,k_n) \in \Z^n_{>0}$\, and\, $k_{i}> k_{i+1}$, $i=1,\dots,n-1$. Set $k_0=k_{n+1}=0$.

\vsk.2>
Let $a,b_1,\dots,b_n,c\in \Z_{>0}$.
  Denote $b=(b_1,\dots,b_n)$
and
\bean
\label{Rbk}
&&
\\
\notag
&&
R_{\bs k} (a,b,c) = 
\\
\notag
&&
=
\prod_{1\leq s\leq r\leq n}\prod_{i=1}^{k_r-k_{r+1}}
\frac{(r-s + b_s+\dots+b_r + (i+s-r-1)c)!}
{(r-s+1 + a_s+b_s+\dots+b_r + (i+s-r+k_s-k_{s-1}-2)c-\delta_{s,1}p)!}
\\
\notag
&&
\times \
(-1)^{\sum_{i=1}^nk_i} \Big(\prod_{i=1}^{k_1}(a_1+(i-1)c)!\Big)
\Big(\prod_{r=1}^{n}\prod_{i=1}^{k_r} \frac{(ic)!}{c!}\Big)
\Big(\prod_{r=2}^{n}\prod_{i=1}^{k_r} (p+(i-k_{r-1}-1)c)! \Big),
\eean
 where $a_1=a$, \,$a_2=\dots=a_{n} =0$;\,  $\delta_{s,1}$ is $1$ if $s=1$ and is zero otherwise.

\vsk.4>

 We say that $a,b_1,\dots,b_n,c\in \Z_{>0}$ are {\it admissible} if
 $a+(k_1-1)c < p-1$ and  for any factorial $x!$ in the right-hand side of  \eqref{Rbk} we have
 $0\leq x<p$. The set of all admissible $(a,b,c)$ 
 is denoted by $\mc A_{\bs k}$. 
\begin{lem}
\label{lem apar}   
The set $\mc A_{\bs k}$ is defined
in $\Z_{> 0}^{n+2}$ 
by the following system of inequalities:
\bean
\label{ine1}
&&
0\leq r-s + b_s+\dots+b_r + (s-r)c, 
\\
\notag
&&
 r-s + b_s+\dots+b_r + (k_r-k_{r+1}+s-r-1)c\leq p-1,
\eean
for $1\leq s\leq r\leq n$;
\bean
\label{ine2}
&&
0\leq r-s +1+b_s+\dots+b_r + (s-r+k_s-k_{s-1}-1)c,
\\
\notag
&&
r-s+1 +b_s+\dots+b_r + (s-r+k_r-k_{r+1}+k_s-k_{s-1}-2)c\leq p-1,
\eean
for $2\leq s\leq r\leq n$;
\bean
\label{ine13}
&&
p\leq r +a+b_1+\dots+b_r + (k_{1} - r)c,
\\
\notag
&&
r +a+b_1+\dots+b_r + (k_r-k_{r+1}+k_{1}-r-1)c < 2p,
\eean
for $1\leq r\leq n$;
\bean
\label{ine14}
a+(k_1-1)c < p-1,\quad b_1\geq p-1-(a+(k_1-1)c), \quad
0 < k_1c < p\,.
\eean
\qed
\end{lem}

\begin{lem}
\label{lem ex p}
Assume that $(a,b,c) \in \mc A_{\bs k}$. Then
\bean
\label{ex p}
b_1\geq p-1-(a+(k_1-1)c), \qquad
b_s\geq (k_{s-1} - k_s+1)c-1, \quad s=2,\dots,n.
\eean
\end{lem}

\begin{proof}
The inequality
$b_s\geq (k_{s-1} - k_s+1)c-1$ for $s=2,\dots,n$ follows from the first inequality in \eqref{ine2} for $r=s$.
The inequality $b_1\geq p-1-(a+(k_1-1)c)$ follows from the first inequality in \eqref{ine13} for $r=1$.
\end{proof}

\begin{example}
Let $n=1$, $\bs k =(k_1)$. Then
\bean
\label{Rbk1}
R_{(k_1)} (a,b_1,c) = 
\prod_{i=1}^{k_1} \frac{(ic)!}{c!}\,
\frac{(a+(i-1)c)!\, (b_1+ (i-1)c)!}
{(1 + a+b_1 + (i+k_1-2)c-p)!}
\eean
and 
$\mc A_{(k_1)}$ consists of $a,b, c\in \Z_{>0}$ such that 
\bean
\label{abc 1 ineq}
&&
a+(k_1-1)c < p-1,
\qquad\qquad
b_1+(k_1-1)c \leq p-1,\qquad k_1c\,\leq \,p-1,
\\
&&
\notag
 p-1\leq a+b_1+(k_1-1)c,
\qquad a+b_1+(2k_1-2)c < 2p-1\ .
\eean

\end{example}

 \subsection{Main theorem}
 
\vsk.2>
Given $\bs k=(k_1,\dots,k_n) \in \Z^n_{> 0}$ introduce 
$k_1+\dots+k_n$ variables 
\bean
\label{def t}
 t=(t^{(1)}, \dots,t^{(n)}), \quad
\on{where}\quad t^{(i)} = (t^{(i)}_1,\dots, t^{(i)}_{k_i}), \quad\, i=1,\dots,n.
\eean
Define  the {\it master polynomial}
\bea
\Phi_{\bs k}(t; a,b,c) = 
\prod_{i=1}^n\Big(
\prod_{j=1}^{k_i} (t^{(i)}_j)^{a_i}(1-t^{(i)}_j)^{b_i}
\prod_{1\leq j<j'\leq k_i}(t^{(i)}_{j}  -t^{(i)}_{j'})^{2c}\Big)
\prod_{i=1}^{n-1} \prod_{j=1}^{k_{i+1}}\prod_{j'=1}^{k_{i}}
 (t^{(i+1)}_{j} - t^{(i)}_{j'})^{p-c}\,.
 \eea
Denote
\bean
\label{def Sk}
S_{\bs k}(a,b,c) = \int_{[\bs k]_p} \Phi_{\bs k}\,d t\, .
\eean
The $\F_p$-integral $S_{\bs k}(a,b,c) $
is called the {\it $\F_p$-Selberg integral of type $A_n$.}

\vsk.2>
Here is our main result.

 \begin{thm}
 \label{thm main n}
 Let $\bs k=(k_1,\dots,k_n) \in \Z^n_{>0}$\,,\,  $k_{i}> k_{i+1}$, $i=1,\dots,n-1$.
 Then for   any $(a,b,c) \in \mc A_{\bs k}$ we have the equality in $\F_p$:
 \bean
\label{main n}
S_{\bs k}(a,b,c)\, = \,R_{\bs k}(a,b,c)\,.
\eean
 \end{thm}

For $n=1$ this theorem is \cite[Theorem 4.1]{RV}. For $n=2$ the theorem is deduced from the case 
of $n=1$  in Section \ref{sec proofs}. More generally, for any $k$  the theorem for  $n=k$  is deduced 
from the theorem for $n=k-1$ similarly, see Section \ref{sec sketch}.

\begin{rem}
Theorem \ref{thm main n} can be extended to the case of $\bs k$ such that
$k_{i}\geq k_{i+1}$, $i=1,\dots,n-1$, but the structure of inequalities in Lemma \ref{lem apar}  will depend on
the appearance of equalities $k_{i} = k_{i+1}$ in $\bs k$, and the proof of Theorem \ref{thm main n} will split into different 
sub-cases. To shorten the exposition we restrict ourselves with $\bs k$ such that $k_i>k_{i+1}$, $i=1,\dots,n-1$.

\end{rem}

\begin{example}
Here is the simplest  $A_2$-type $\F_p$-Selberg integral formula with $k_1=k_2=1$.

\begin{thm}
\label{thm 3-11}
Assume that $a,b_1,b_2, c$ are integers such that
\bea
&&
0\leq a <p,
\quad
 0< c\leq p,\quad
 0\leq b_2-c+1<p,
\\
&&
0\leq b_1+b_2-c+1<p,
\quad 
p-1\leq a+b_1+b_2 -c +1< 2p-1.
\eea
Then in $\F_p$ we have
\bean
\label{3-11}
&&
\int_{[1;1]_p}
t^{a}(1-t)^{b_1}(s-t)^{p-c} (1-s)^{b_2} dt\,ds
= \ 
\frac{a!\,(b_1+b_2-c+1)!}{(a+b_1+b_2-c+2-p)!}\,
\frac{(p-c)!\,(b_2)!}{(b_2-c+1)!}\,.
\eean

\end{thm}

\begin{proof}
Change variables $s=t+(1-t)v$, then the $\F_p$-integral becomes equal to
\bea
\int_{[1,1]_p}
t^{a}(1-t)^{b_1+b_2-c+1} v^{p-c} (1-v)^{b_2} dt\,dv .
\eea
Applying the $\F_p$-beta integral formula we obtain the theorem.
\end{proof}

The simplest $A_3$-type $\F_p$-Selberg integral formula  $k_1=k_2=k_3=1$
is given by the next theorem.

\begin{thm}
\label{thm 4-111}

Let $a,b_1,b_2,b_3,c$  be integers such that 
all factorials in the right-hand side of formula \eqref{4-111} are factorials of nonnegaive integers less than $p$.
Then in $\F_p$ we have
\bean
\label{4-111}
&&
\int_{[1;1;1]_p}
t^{a}(1-t)^{b_1}(s-t)^{p-c} (1-s)^{b_2}(u-s)^{p-c} (1-u)^{b_2} dt\,ds\,du
\\
\notag
&&
\phantom{aa}
= \ -\
\frac{a!\,(b_1+b_2+b_3-2c+2)!}{(a+b_1+b_2+b_3-2c+3-p)!}\,
\frac{(p-c)!\,(b_2+b_3-c+1)!}{(b_2+b_3-2c+2)!}\,
\frac{(p-c)!\,(b_3)!}{(b_3-c+1)!}\,.
\eean

\end{thm}

\begin{proof}
The proof is the same as of the previous theorem.
\end{proof}

The versions of identities \eqref{3-11}, \eqref{4-111} over complex numbers see in 
\cite[Theorem 1]{MuV}.
\end{example}

\subsection{Relation between the $\F_p$-Selberg integrals of types $A_{n-1}$- and $A_n$}

\begin{thm}
\label{thm ind}
Let  $n>1$ and $\bs k=(k_1,\dots,k_n)$, $\bs k'=(k_1,\dots,k_{n-1})$, $b=(b_1,\dots,b_n)$,
$b'=(b_1,\dots,b_{n-1})$. 
Assume  that formula \eqref{main n} holds for the $\F_p$-Selberg integral 
$S_{[\bs k']_p}(a,b',c)$ of type $A_{n-1}$. 
Also assume that $b_n=(k_{n-1}-k_n+1)c-1$.
Then formula \eqref{main n} holds for
the $\F_p$-Selberg integral $S_{[\bs k]_p}(a,b,c)$ of type $A_{n}$.
\end{thm}

\begin{proof}
Under the assumption  $b_n=(k_{n-1}-k_n+1)c-1$ all variables $(t^{(n)}_j)$
in $\Phi_{[\bs k]_p}(\bs t;a,b,c)$ are used  to reach
the monomial $\prod_{j=1}^{k_n} (t^{(n)}_j)^{k_{n-1}p-1}$ in the calculation of the $\F_p$-integral
$S_{[\bs k]_p}(a,b,c)$.  The remaining free variables 
$(t^{(i)}_j)$ with $i<n$ all belong to the factor
\\
$\Phi_{[\bs k']_p}(t^{(1)},\dots,t^{(n-1)};a,b',c)$ of  $\Phi_{[\bs k]_p}(\bs t;a,b,c)$
and are used to calculate the coefficient of
$\prod_{j=1}^{k_1} (t^{(1)}_j)^{p-1}\prod_{i=2}^{n-1}\prod_{j=1}^{k_i} (t^{(i)}_j)^{k_{i-1}p-1}$.

More precisely, under the assumptions of the theorem we have
\bea
&&
S_{\bs k}(a,b,c)  = 
  (-1)^{b_nk_n+ck_n(k_n-1)/2}\frac{(k_nc)!}{(c!)^{k_n}}
     \,S_{\bs k'}(a',b',c),
\eea
where $ (-1)^{b_nk_n+ck_n(k_n-1)/2}\frac{(k_nc)!}{(c!)^{k_n}}$ is the coefficient of 
  $\prod_{j=1}^{k_n} (t^{(n)}_j)^{k_{n-1}p-1}$ in the expansion of
 \bea
 \prod_{j=1}^{k_n} \prod_{j'=1}^{k_{n-1}} (t^{(n)}_j  - t^{(n-1)}_{j'})^{p-c}
\prod_{1\leq j<j'\leq k_n}(t^{(n)}_j -t^{(n)}_{j'})^{2c} \,,
\eea
see Dyson's formula. We have $ (-1)^{b_nk_n+ck_n(k_n-1)/2} = (-1)^{ (k_{n-1}k_n-k_n(k_n+1)/2)c-k_n}$.
Hence
\bea
S_{\bs k}(a,b,c)
&  =&
(-1)^{ (k_{n-1}k_n-k_n(k_n+1)/2)c-k_n}\frac{(k_nc)!}{(c!)^{k_n}}
     \,S_{\bs k'}(a,b',c)
\\
&=&
(-1)^{ (k_{n-1}k_n-k_n(k_n+1)/2)c-k_n}\frac{(k_nc)!}{(c!)^{k_n}}
     \,R_{\bs k'}(a,b',c)\,,
\eea
where $S_{\bs k'}(a,b',c)
=
R_{\bs k'}(a,b',c)$ holds by assumptions.
To prove the theorem we need to show that 
\bea
R_{[\bs k]_p}(a, (b',b_n),c) = (-1)^{ (k_{n-1}k_n-k_n(k_n+1)/2)c-k_n}\frac{(k_nc)!}{(c!)^{k_n}}
     \,R_{\bs k'}(a,b',c)\,.
\eea
Indeed we have
\bea
&&
R_{\bs k}(a,(b',b_n),c) = \,R_{\bs k'}(a,b',c)\,
\prod_{i=1}^{k_n} \frac{(ic)!}{c!}
\prod_{i=1}^{k_n} (p+(i-k_{n-1}-1)c)! 
\\
&&
\times
\prod_{1\leq s\leq n}\prod_{i=1}^{k_n}
\frac{(n-s + b_s+\dots+b_n + (i+s-n-1)c)!}
{(n-s+1 + a_s+b_s+\dots+b_n + (i+s-n+k_s-k_{s-1}-2)c-\delta_{s,1}p)!}
\\
&&
\times
\prod_{1\leq s\leq n-1}\prod_{i=1}^{k_{n}}
\frac{(n-s + a_s+b_s+\dots+b_{n-1} + (i+s-n+k_s-k_{s-1}-3)c-\delta_{s,1}p)!}
{(n-1-s + b_s+\dots+b_{n-1} + (i+s-r-1)c)!}
\eea
\bea
&&
= \,R_{\bs k'}(a,b',c)\,
\prod_{i=1}^{k_n} \frac{(ic)!}{c!}
\,\prod_{i=1}^{k_n}
\frac{( b_n + (i-1)c)!\,(p+(i-k_{n-1}-1)c)!}
{(1 +b_n + (i+k_n-k_{n-1}-2)c)!}\,.
\\
&&
= \,R_{\bs k'}(a,b',c)\,
\prod_{i=1}^{k_n} \frac{(ic)!}{c!}
\,\prod_{i=1}^{k_n}
\frac{( (i+k_{n-1}-k_n)c-1)!\,(p+(i-k_{n-1}-1)c)!}
{((i-1)c))!}\,.
\\
&&
= \,R_{\bs k'}(a,b',c)\,
(-1)^{ (k_{n-1}k_n-k_n(k_n+1)/2)c-k_n}
\frac{(k_nc)!}{(c!)^{k_n}}\,,
\eea
where in the last step we use the cancellation Lemma \ref{lem ca}.
The theorem is proved.
\end{proof}

\begin{cor}
\label{cor ind}
Let $n>1$, $\bs k=(k_1,\dots,k_n)$,  and $(a,b_1,c) \in\mc A_{(k_1)}$. Let
$b=(b_1,\dots,b_n)$, where $b_i=(k_{i-1}-k_i+1)c-1$ for $i=2,\dots,n$.
Then formula \eqref{main n} holds for 
the $\F_p$-Selberg integral $S_{[\bs k]_p}(a,b,c)$ of type $A_{n}$.
\end{cor}

\begin{proof} Formula \eqref{main n} for the $\F_p$-Selberg integrals of type $A_1$ is proved in \cite{RV}.
Hence the corollary follows  from Theorem \ref{thm ind} by induction on $n$.
\end{proof}

\section{The $A_2$-type Selberg integral over $\C$}
\label{sec 4}

In this section we formulate the $A_2$-type Selberg integral formula over $\C$, formulated and proved in \cite{TV4},
and sketch the proof of the formula,   following  \cite{TV4}.
In Section \ref{sec proofs} we adapt this proof to prove the $A_2$-type $\F_p$-Selberg integral formula,
that is, formula \eqref{main n} for $n=2$.

\subsection{The $A_2$-formula over $\C$}

For $k_1\geq k_2\geq 0$ let $t=(t_1,\dots,t_{k_1})$, $s=(s_1,\dots,s_{k_2})$.
Define the {\it master function}
\bean
\label{mf}
&&
\Phi(t;s) = \prod_{i=1}^{k_1} t_i^{\al_1-1}(1-t_i)^{\beta_1-1}\, \prod_{j=1}^{k_2}(1-s_j)^{\beta_2-1} 
\,\prod_{i=1}^{k_1}\prod_{j=1}^{k_2}  |s_j-t_i|^{-\ga}
\\
\notag
&&
\phantom{aaaaaaaaaaaaaa}
\times\,
\prod_{1\leq i<i'\leq k_1} (t_i-t_{i'})^{2\ga}
\prod_{1\leq j<j'\leq k_2} (s_j-s_{j'})^{2\ga}\, 
\eean
and the integral
\bean
\label{tilde S}
\tilde S(\al,\beta_1,\beta_2,\ga) = \int_{C^{k_1,k_2}[0,1]}\, \Phi(t;s)\,dt\,ds\,,
\eean
where $C^{k_1,k_2}[0,1]$ is the integration cycle, defined in \cite[Section 3]{TV4}, also see its definition in \cite{Wa1,Wa2, FW}.
The explicit description of this cycle is of no importance in this paper.

\begin{thm} [{\cite[Theorem 3.3]{TV4}}]
\label{thm A2C} 

Let $a,b_1,b_2,c\in\C$ such that $\on{Re}(a),\,\on{Re}(b_1),\,\on{Re}(b_2) > -1$,
$\on{Re}(c) <0$ and $|\on{Re}(c)|$ sufficiently small. Then
\bean
\label{A2C}
&&
\tilde S(\al,\beta_1,\beta_2,\ga)  
 = \prod_{i=1}^{k_1-k_2}
\frac{\Ga(\beta_1+(i-1)\ga)}
{\Ga(\al+\beta_1+(i+k_1-2)\ga)}
 \\
 \notag
 &&
\times\,
\prod_{i=1}^{k_2}\
\frac{\Ga(\beta_2+(i-1)\ga)}
{\Ga(1+\beta_2+(i+k_2-k_1-2)\ga)}\,
\
\frac{\Ga(\beta_1+\beta_2+(i-2)\ga)}
{\Ga(\al+\beta_1+\beta_2+(i+k_2-3)\ga)}
\\
&&
\notag
\times \
\prod_{i=1}^{k_2}
\frac{\Ga(1+(i-k_1-1)\ga)\,\Ga(i\ga)}
{\Ga(\ga)}\
\prod_{i=1}^{k_1}
\frac{\Ga(\al+(i-1)\ga)\,\Ga(i\ga)}
{\Ga(\ga)}\,.
\eean
 \end{thm}

In the next Sections \ref{sec wf}, \ref{sec reps C}
we sketch the proof of formula \eqref{A2C}
 following  \cite{TV4}.

\subsection{Weight functions}
\label{sec wf}

To evaluate  $\tilde S(\al,\beta_1,\beta_2,\ga)$ we introduce a collection 
of new integrals 
$J_{l_1,l_2,m}(\al,\beta_1,\beta_2,\ga)$, which also can be evaluated explicitly, see \cite{TV4}.

\vsk.2>
For a function $f(t_1,\dots,t_k)$ set
\bea
\Sym_{t_1,\dots,t_k} f(t_1,\dots,t_k)\, =\,\frac1{k!}\, \sum_{\si\in S_k} f(t_{\si_1},\dots, t_{\si_k})\,.
\eea

Given $k_1\geq k_2\geq 0$, we say that a triple of nonnegative integers $(l_1,l_2, m)$
is {\it allowable}  if  $l_1\leq k_1-k_2+l_2, \,\,l_2\leq k_2$ and $m\leq \on{min}(l_1,l_2)$.
For any allowable triple $(l_1 , l_2 , m)$  define the weight function
\bea
&&
W_{l_1,l_2,m}(t_1,\dots,t_{k_1};s_1,\dots,s_{k_2}) = 
\\
&&
\phantom{aaa}
=\,\Sym_{t_1,\dots,t_{k_1}}\Sym_{s_1,\dots,s_{k_2}}\Big(
\prod_{a=1}^{l_1} t_a\prod_{a=l_1+1}^{k_1}(1-t_a)
\prod_{b=1}^m\frac{1-s_b}{s_b-t_b}\prod_{b=l_2+1}^{k_2}\frac{1-s_b}{s_b-t_{b+k_1-k_2}}\Big)
\eea
and the integral
\bea
 J_{l_1,l_2,m}(\al,\beta_1,\beta_2,\ga) =  \int_{C^{k_1,k_2}[0,1]}\, \Phi(t;s)\,
 W_{l_1,l_2,m}(t;s)\,dt\,ds\,.
 \eea
In particular,  
\bean
\label{SJ}
J_{0,k_2,0}(\al,\beta_1,\beta_2,\ga) = \tilde S(\al,\beta_1+1,\beta_2,\ga).
\eean

\subsection{Representations of $\slth$}
\label{sec reps C}

Consider the complex Lie algebra $\slth$ with standard generators $f_1,f_2, e_1,e_2,h_1,h_2$,
simple roots $\si_1,\,\si_2$, fundamental weights  $\om_1,\,\om_2$.
Let $V_{\la_1}$, $V_{\la_2}$ be the irreducible $\slth$-modules with highest weights
\bea
\la_1= -\frac{\al}{\ga}\,\om_1 \,,
\qquad
\la_2= -\frac{\beta_1}{\ga}\,\om_1 - \frac{\beta_2}{\ga}\,\om_2 \,
\eea
and highest weight vectors $v_1,\,v_2$. For $k_1\geq k_2\geq 0$ consider the weight subspace
\\
$V_{\la_1}\ox V_{\la_2}[\la_1+\la_2-k_1\si_1-k_2\si_2]$ of the tensor product 
 $V_{\la_1}\ox V_{\la_2}$ and the singular weight subspace 
 $\Sing\, V_{\la_1}\ox V_{\la_2}[\la_1+\la_2-k_1\si_1-k_2\si_2]$ consisting of the vectors
$w\in V_{\la_1}\ox V_{\la_2}[\la_1+\la_2-k_1\si_1-k_2\si_2]$ such that $e_1w=0$, $e_2 w=0$.
A basis of $V_{\la_1}\ox V_{\la_2}[\la_1+\la_2-k_1\si_1-k_2\si_2]$ is formed by the vectors
\bea
v_{l_1,l_2,m} = \frac{f_1^{k_1-k_2-l_1+l_2}[f_1,f_2]^{k_2-l_2}v_1\ox f_1^{l_1-m}[f_1,f_2]^{m}f_2^{l_2-m}v_2}
{(k_1-k_2-l_1+l_2)!\,(k_2-l_2)!\,(l_1-m)!\,m!\,(l_2-m)!}
\eea
labeled by allowable triples $(l_1,l_2,m)$. 
It is known from the theory of KZ equations  that the vector
 \bea
 J = \sum_{l_1,l_2,m} (-1)^{l_1} J_{l_1,l_2,m} (\al,\beta_1,\beta_2,\ga)\, v_{l_1,l_2,m}
 \eea
 is  a singular vector, see \cite[Theorem 2.4]{M}, \cite[Corollary 10.3]{MaV}, cf. \cite{RSV}.
 
 \vsk.2> 
 The singular vector equations $e_1J=0, \,e_2J=0$ are calculated with the help of the formulas:
\bea
&&
h_1 v_1 = -\frac{\al}{\ga}\,v_1,\qquad  h_2 v_1 = 0,
\qquad
h_1 v_2 = -\frac{\beta_1}{\ga}\,v_2,\qquad  h_2 v_2 = -\frac{\beta_2}{\ga}\,v_2\,,
\\
&&
[h_1,f_1]=-2f_1,\quad \ [h_1,f_2]=f_2, \quad [h_2,f_1]=f_1,
\quad\ \  \ \ \ \  [h_2,f_2]=-2f_2,
\\
&&
[e_1,f_1]=h_1, \qquad\ \ \ \ \ [e_1,f_2]\ =\ [e_2,f_1]\ =\ 0,\qquad \ \ \  \ \ [e_2,f_2]=h_2,
\eea
\bea
[h_1,[f_1,f_2]]=-[f_1,f_2],
\quad
[h_2,[f_1,f_2]]=-[f_1,f_2],
\quad
[e_1,[f_1,f_2]]=f_2,
\quad
[e_2,[f_1,f_2]]=-f_1.
\eea 
Here are some of the singular vector relations.

\begin{thm} [cf. {\cite[Theorem 5.2]{TV4}}]
\label{thm svr}
We have
\bean
\label{OLO}
J_{0,l_2,0} = (-1)^{l_2}J_{0,0,0}\prod_{i=0}^{l_2-1} \frac{(k_1-k_2+i+1)\ga}{\beta_2+i\ga}\,.
\eean

\end{thm}

\begin{proof}
We have
\bea
e_2 
\frac{f_1^{k_1-k_2+i}[f_1,f_2]^{k_2-i}v_1\ox f_2^{i}v_2}
{(k_1-k_2+i)!\,(k_2-i)!\,i!} 
=
&-&
(k_1-k_2+i+1)
\frac{f_1^{k_1-k_2+i+1}[f_1,f_2]^{k_2-i-1}v_1\ox f_2^{i}v_2}
{(k_1-k_2+i+1)!\,(k_2-i-1)!\,i!}
\\
&+&  \Big(-\frac{\beta_2}{\ga} - i+1\Big)
\frac{f_1^{k_1-k_2+i}[f_1,f_2]^{k_2-i}v_1\ox f_2^{i-1}v_2}
{(k_1-k_2+i)!\,(k_2-i)!\,(i-1)!} \,.
\eea
Calculating the coefficient of 
$\frac{f_1^{k_1-k_2+i+1}[f_1,f_2]^{k_2-i-1}v_1\ox f_2^{i}v_2}
{(k_1-k_2+i)!\,(k_2-i)!\,i!}$ in $e_2J=0$ we obtain
\bean
\label{JJ=0}
(k_1-k_2+i+1)\ga\,J_{0,i,0} +
(\beta_2 +i\ga)\,J_{0,i+1,0}  = 0.
\eean
This implies the theorem.
\end{proof}

Hence
\bean
\label{JJ}
J_{0,k_2,0}(\al,\beta_1,\beta_2,\ga)
=(-1)^{k_2} J_{0,0,0}(\al,\beta_1,\beta_2,\ga) \prod_{i=0}^{k_2-1} \frac{(k_1-k_2+i+1)\ga}{\beta_2+i\ga}\,.
\eean
Combining \eqref{SJ} and \eqref{JJ} we observe that formula \eqref{A2C} is equivalent to the formula
\bean
\label{JJ2}
&&
J_{0,0,0}(\al,\beta_1,\beta_2,\ga)  
 = 
 \prod_{i=1}^{k_1-k_2}
\frac{\Ga(1+\beta_1+(i-1)\ga)}
{\Ga(1+\al+\beta_1+(i+k_1-2)\ga)}
 \\
 \notag
 &&
\times\,
\prod_{i=1}^{k_2}\
\frac{\Ga(1+\beta_2+(i-1)\ga)}
{\Ga(1+\beta_2+(i+k_2-k_1-2)\ga)}\,
\
\frac{\Ga(1+\beta_1+\beta_2+(i-2)\ga)}
{\Ga(1+\al+\beta_1+\beta_2+(i+k_2-3)\ga)}
\\
&&
\notag
\times \
\prod_{i=1}^{k_2}
\frac{\Ga((i-k_1-1)\ga)\,\Ga(i\ga)}
{\Ga(\ga)}\
\prod_{i=1}^{k_1}
\frac{\Ga(\al+(i-1)\ga)\,\Ga(i\ga)}
{\Ga(\ga)}\,.
\eean
Denote by $R_{0,0,0}(\al,\beta_1,\beta_2,\ga)$ the right-hand side of \eqref{JJ2}.

\vsk.2>
To prove \eqref{JJ2} we use the following observation.
The weight subspace 
$V_{\la_1}[\la_1-k_1\si_1-k_2\si_2]\subset V_{\la_1}$ 
is one-dimensional with a basis vector
\bea
v_{0,0,0}\,=\,
\frac{f_1^{k_1-k_2}[f_1,f_2]^{k_2}v_1\ox v_2}{(k_1-k_2)!\,(k_2)!} \,.
 \eea
By \cite[Theorem 5.1]{MaV}
the vector-valued  function $J_{0,0,0}(\al,\beta_1,\beta_2,\ga) v_{0,0,0}$
satisfies the dynamical difference equations introduced in \cite{TV3}, 
\bean
\label{dyne} 
&&
J_{0,0,0}(\al,\beta_1-1,\beta_2,\ga) v_{0,0,0}
= J_{0,0,0}(\al,\beta_1,\beta_2,\ga) \,\Bbb B_1 v_{0,0,0}\,,
\\
\label{dyne2}
&&
J_{0,0,0}(\al,\beta_1,\beta_2-1,\ga) v_{0,0,0}
= J_{0,0,0}(\al,\beta_1,\beta_2,\ga)\, \Bbb B_2 v_{0,0,0}\,,
\eean
where $\Bbb B_1,\,\Bbb B_2$ are certain linear operators acting on 
$V_{\la_1}$ and preserving  the weight decomposition of 
$V_{\la_1}$,  see formulas for these operators
in the example in \cite[Section 7.1]{MaV}
and in \cite[Section 3.1]{MaV}, also see \cite[Formula (8)]{TV3}.

Written explicitly equations \eqref{dyne}, \eqref{dyne2} give us the difference equations for the
scalar  function 
$J_{0,0,0}(\al,\beta_1,\beta_2,\ga)$ with respect to the shift of
the variables $\beta_1 \to \beta_1-1$ and $\beta_2 \to\beta_2-1$,
\bean
\label{Dyne1} 
&&
J_{0,0,0}(\al,\beta_1-1,\beta_2,\ga) 
= J_{0,0,0}(\al,\beta_1,\beta_2,\ga) \, 
\prod_{i=1}^{k_1-k_2}
\frac{\al+\beta_1+(i+k_1-2)\ga}
{\beta_1+(i-1)\ga}
\\
\notag
&&
\phantom{aaaaaaaaaaaaaaaaaa}
\times \, \prod_{i=1}^{k_2}\
\frac
{\al+\beta_1+\beta_2+(i+k_2-3)\ga}
 {\beta_1+\beta_2+(i-2)\ga} \,,
\\
\label{Dyne2} 
&&
J_{0,0,0}(\al,\beta_1,\beta_2-1,\ga) 
=
 J_{0,0,0}(\al,\beta_1,\beta_2,\ga)\, 
\\
  \notag
 &&
\phantom{aaaaaaaaaa}
\times \,
\prod_{i=1}^{k_2}\
\frac
{\beta_2+(i+k_2-k_1-2)\ga}
{\beta_2+(i-1)\ga}\
\frac
{\al+\beta_1+\beta_2+(i+k_2-3)\ga}
{\beta_1+\beta_2+(i-2)\ga}
 \,.
\eean

The difference equations for $ J_{0,0,0}(\al,\beta_1,\beta_2,\ga)\, 
$ are the same
as the difference equations for the function 
$R_{0,0,0}(\al,\beta_1,\beta_2,\ga)$  with respect to the shift of
the variables $\beta_1 \to \beta_1-1$ and $\beta_2 \to\beta_2-1$.
Therefore, 
the functions $J_{0,0,0}(\al,\beta_1,\beta_2,\ga)$ and 
$R_{0,0,0}(\al,\beta_1,\beta_2,\ga)$ 
are proportional up to a periodic function of $\beta_1,\beta_2$. 
The periodic function can be fixed by comparing asymptotics 
as $\on{Re} \beta_1 \to \infty$,
$\on{Re} \beta_2 \to \infty$.  This finishes the proof in \cite{TV4}
 of formulas \eqref{JJ2} and \eqref{A2C}.

\section{The $A_2$-type Selberg integrals over $\F_p$}
 \label{sec proofs}

 \subsection{Relations between $\F_p$-integrals}
 
For $\bs k=(k_1,k_2)$, $k_1> k_2> 0$ and integers 
\\
$0\,<\,a,b_1,b_2,c\,<\,p$ define the master polynomial
\bean
\label{mp2}
&&
\Phi_{\bs k}(t;s; a,b_1,b_2,c) = \prod_{i=1}^{k_1} t_i^{a}(1-t_i)^{b_1}\, \prod_{j=1}^{k_2}(1-s_j)^{b_2} 
\,\prod_{i=1}^{k_1}\prod_{j=1}^{k_2}  (s_j-t_i)^{p-c}
\\
\notag
&&
\phantom{aaaaaaaaaaaaaa}
\times\,
\prod_{1\leq i<i'\leq k_1} (t_i-t_{i'})^{2c}
\prod_{1\leq j<j'\leq k_2} (s_j-s_{j'})^{2c}\,
\eean
and the $\F_p$-integral
\bean
\label{Sint}
S_{\bs k}(a,b_1,b_2,c) = \int_{[\bs k]_p}\, \Phi_{\bs k}(t;s; a,b_1,b_2,c)\,dt\,ds\,,
\eean
where the $p$-cycle $[\bs k]_p$ is defined in  \eqref{def  bsk}.   
This is the $A_2$-type $\F_p$-Selberg integral, see \eqref{def Sk}.

For an allowable triple $(l_1,l_2,m)$ define the $\F_p$-integral
\bean
\label{I}
\phantom{aa}
I_{l_1,l_2,m}(a,b_1,b_2,c) =  
\int_{[\bs k]_p}\, \Phi_{\bs k}(t;s; a,b_1,b_2,c)\, 
\frac{W_{l_1,l_2,m}(t;s)}
{\prod_{i=1}^{k_1}t_i(1-t_i)\prod_{j=1}^{k_2}(1-s_j)}
\,dt\,ds\,,
\eean
where $ W_{l_1,l_2,m}(t;s)$ is the weight function defined in Section \ref{sec wf}.

Clearly we have
\bean
\label{IS}
I_{0,k_2,0}(a,b_1,b_2,c) = S_{\bs k}(a-1,b_1,b_2-1,c).
\eean

Denote
\bean
\label{Df} 
&&
B_0(a,b_1,b_2,c)
=(-1)^{k_2}  
\prod_{i=0}^{k_2-1} \frac{(k_1-k_2+i+1)c}{b_2+ic}\,,
\\
&&
\notag
B_1(a,b_1,b_2,c)
= 
\prod_{i=1}^{k_1-k_2}
\frac{a+b_1+(i+k_1-2)c}
{b_1+(i-1)c}
 \, \prod_{i=1}^{k_2}\
\frac
{a+b_1+b_2+(i+k_2-3)c}
 {b_1+b_2+(i-2)c} \,,
\\
\notag
&&
B_2(a,b_1,b_2,c)
= 
\prod_{i=1}^{k_2}\
\frac
{b_2+(i+k_2-k_1-2)c}
{b_2+(i-1)c}\
\frac
{a+b_1+b_2+(i+k_2-3)c}
{b_1+b_2+(i-2)c}
 \,.
\eean

\begin{thm}
\label{thm relp}
Assume that $k_1<p$.
\begin{enumerate}

\item[(i)]
Assume that every factor in $B_0$ in the numerator or denominator is a
nonzero element of $\F_p$.  Then
\bean
\label{II}
I_{0,k_2,0}(a,b_1,b_2,c)
= B_0(a,b_1,b_2,c)\,I_{0,0,0}(a,b_1,b_2,c)\,.
\eean

\item[(ii)]
Assume that every factor in $B_1$ in the numerator or denominator is a
nonzero element of $\F_p$ and  $b_1>1$, then
\bean
\label{B1}
&&
I_{0,0,0}(a,b_1-1,b_2,c)\,=\, B_1(a,b_1,b_2,c) \,I_{0,0,0}(a,b_1,b_2,c).
\eean

\item[(iii)]
Assume that every factor in $B_2$ in the numerator or denominator is a
nonzero element of $\F_p$ and $b_2>1$, then
\bean
\label{B2}
&&
I_{0,0,0}(a,b_1,b_2-1,c)\,=\, B_2(a,b_1,b_2,c) \,I_{0,0,0}(a,b_1,b_2,c).
\eean
\end{enumerate}

\end{thm}

\begin{proof} Equation \eqref{II} is an $\F_p$-analog of equation
\eqref{JJ} and its proof is analogous to the proof of equation \eqref{JJ}. 

More precisely, consider the  complex Lie algebra $\slth$ 
with standard generators $f_1,f_2$, $e_1,e_2$, $h_1,h_2$,
simple roots $\si_1,\,\si_2$, fundamental weights  $\om_1,\,\om_2$.
Let $V_{\la_1}$, $V_{\la_2}$ be the irreducible $\slth$-modules with highest weights
\bea
\la_1= -\frac{a}{c+p}\,\om_1 \,,
\qquad
\la_2= -\frac{b_1}{c+p}\,\om_1 - \frac{b_2}{c+p}\,\om_2 \,
\eea
and highest weight vectors $v_1,\,v_2$. 
The module $V_{\la_1}$ has a basis $(f_1^{r_1}[f_1,f_2]^{r_{2}}v_1)$ labeled by nonnegative integers $r_1,r_2$
and the module $V_{\la_2}$ has a basis $(f_1^{r_1}[f_1,f_2]^{r_{2}}f_2^{r_3}v_2)$ labeled by nonnegative integers $r_1,r_2, r_3$.
For every generator of $\slth$ the matrix of its action on
$V_{\la_1}$ or on $V_{\la_2}$ in these bases  is a polynomial in $ -\frac{a}{c+p}$\,
$-\frac{b_1}{c+p},$ $ - \frac{b_2}{c+p}$ with integer coefficients.

Consider the Lie algebra 
$\frak{sl}_3$ over the field $\F_p$. Let $V_{\la_1}^{\F_p}$ be the vector space over $\F_p$ with 
 basis 
$(f_1^{r_1}[f_1,f_2]^{r_{2}}v_1)$ labeled by nonnegative integers $r_1,r_2$
and with the action of $\frak{sl}_3$ defined by the same formulas as on $V_{\la_1}$ 
but reduced modulo $p$. 
Similarly we define the $\frak{sl}_3$-module $V_{\la_2}^{\F_p}$. 

\vsk.2>
Recall  $\bs k=(k_1,k_2)$, $k_1> k_2> 0$.
Consider the weight subspace
$V_{\la_1}^{\F_p}\ox V_{\la_2}^{\F_p}[\la_1+\la_2-k_1\si_1-k_2\si_2]$ of the tensor product 
 $V_{\la_1}^{\F_p}\ox V_{\la_2}^{\F_p}$. This weight subspace has a basis 
  formed by the vectors
\bea
v_{l_1,l_2,m} = \frac{f_1^{k_1-k_2-l_1+l_2}[f_1,f_2]^{k_2-l_2}v_1\ox f_1^{l_1-m}[f_1,f_2]^{m}f_2^{l_2-m}v_2}
{(k_1-k_2-l_1+l_2)!\,(k_2-l_2)!\,(l_1-m)!\,m!\,(l_2-m)!}
\eea
labeled by allowable triples $(l_1,l_2,m)$. 

\begin{lem}
\label{lem sing p}
The vector 
\bea
 I = \sum_{l_1,l_2,m} (-1)^{l_1} I_{l_1,l_2,l_m} (a,b_1,b_2,c)\, v_{l_1,l_2,m}
 \eea
 is  a singular vector of $V_{\la_1}^{\F_p}\ox V_{\la_2}^{\F_p}$, that is, $e_1I=0, \,e_2I=0$.
\end{lem}

\begin{proof}
Equations  $e_1I=0, \,e_2I=0$ are $\F_p$-analogs of equations $e_1J=0$, \, $e_2J=0$ over  $\C$.

For $i=1,2$, the vector $e_iJ$ is the integral of a certain differential  $k_1+k_2$-form $\mu_i$. It is shown in  
\cite[Theorems 6.16.2]{SV2},  \cite[Theorem 2.4]{M} that $\mu_i= \on{d}\! \nu_i$, where $\nu_i$ is  some explicitly written  
differential $k_1+k_2-1$-form. This implies  $e_iJ=0$ by Stokes' theorem.

The vector $e_iI$ is the $\F_p$-integral of the same $\mu_i$  reduced modulo $p$.
It is explained in \cite[Section 4]{SV2} 
that the differential form $\nu_i$ also can be reduced modulo $p$ and this implies that the  $\F_p$-integral
$e_iI$   is zero by Lemma \ref{lem St}. Cf. the proof of \cite[Theorem 2.4]{SV2}.
\end{proof}

Lemma \ref{lem sing p} implies the equations
\bean
\label{II=0}
(k_1-k_2+i+1)c\,I_{0,i,0} +
(b_2 +ic)\,I_{0,i+1,0}  = 0
\eean
for $i=0,\dots,k_2-1$, similarly to the proof of equations \eqref{JJ=0}. 
The iterated application of equation \eqref{II=0} implies equation \eqref{II}.

\vsk.2>
The proof of equations \eqref{B1}, \eqref{B2} is parallel to the proof of equations \eqref{Dyne1}, \eqref{Dyne2}. 
We prove  \eqref{B1}. The proof of \eqref{B2} is similar.

Equation \eqref{Dyne1} follows from equation \eqref{dyne}:
\bean
\label{dyne11} 
J_{0,0,0}(\al,\beta_1-1,\beta_2,\ga) v_{0,0,0}
- J_{0,0,0}(\al,\beta_1,\beta_2,\ga) \,\Bbb B_1 v_{0,0,0}\,=\,0
\eean
and equation $\Bbb B_1 v_{0,0,0}= B_1 v_{0,0,0}$ in $V_{\la_1}$. 
The explicit formulas for $\Bbb B_1$ show that under the assumptions of Theorem \ref{thm relp}
the action of $\Bbb B_1$ on $v_{0,0,0}$ is well-defined modulo $p$ and gives
the same result
$\Bbb B_1 v_{0,0,0}= B_1 v_{0,0,0}$ but in $V_{\la_1}^{\F_p}$.

The proof  of \eqref{dyne11} in \cite{MaV}
goes as follows. The left-hand side of \eqref{dyne11} is a vector-valued integral of a suitable
differential $k_1+k_2$-form $\mu$.
It is shown in  
\cite[Theorem 5.1]{M} that $\mu= \on{d}\! \nu$, where $\nu$ is  some explicitly written  
differential $k_1+k_2-1$-form. This implies \eqref{dyne11} by Stokes' theorem.

The $p$-analog of the left-hand side of \eqref{dyne11} is 
the element
\bean
\label{dyne111} 
I_{0,0,0}(a,b_1-1,b_2,c) v_{0,0,0}
- I_{0,0,0}(a,b_1,b_2,c) \,\Bbb B_1 v_{0,0,0}\,\in \,V_{\la_1}^{\F_p}\,.
\eean
This element 
is the $\F_p$-integral of the same $\mu$ reduced modulo $p$.
It is explained in \cite[Section 4]{SV2} 
that the differential form $\nu$ also can be reduced modulo $p$ and this implies that the  $\F_p$-integral
in the left-hand side of \eqref{dyne111}   equals  zero by Lemma \ref{lem St}. Hence equation \eqref{B1} is proved
and Theorem \ref{thm relp} is proved.
\end{proof}

\subsection{Theorem \ref{thm main n} for $n=2$}

Recall the set of admissible parameters $\mc A_{\bs k}$ introduced in Section \ref{sec adp}
for $\bs k = (k_1,k_2)$, $k_1> k_2> 0$.

\begin{lem}
\label{lem S-1}
Assume that $(a,b_1-1,b_2,c), (a,b_1,b_2,c) \in \mc A_{\bs k}$. Then
\bean
\label{S-1}
S_{\bs k}(a,b_1-1,b_2,c) 
&=& 
S_{\bs k}(a,b_1,b_2,c)
\prod_{i=1}^{k_1-k_2}
\frac
{1+a+b_1+(i+k_1-2)c-p}
 {b_1+(i-1)c}
 \\
 \notag
 &\times &\prod_{i=1}^{k_2}
\frac
{2+a+b_1+b_2+(i+k_1-3)c-p}
{1+b_1+b_2+(i-2)c}\,.
\eean
Assume that $(a,b_1,b_2-1,c), (a,b_1,b_2,c) \in \mc A_{\bs k}$. Then
\bean
\label{S-2}
S_{\bs k}(a,b_1,b_2-1,c) 
&=& 
S_{\bs k}(a,b_1,b_2,c)
\prod_{i=1}^{k_2}
\frac{1+b_2+(i+k_2-k_1-2)c}
{b_2+(i-1)c}
 \\
 \notag
 &\times &\prod_{i=1}^{k_2}
\frac
{2+a+b_1+b_2+(i+k_1-3)c-p}
{1+b_1+b_2+(i-2)c}\,.
\eean

\end{lem}

\begin{proof} 
The lemma follows from formulas \eqref{IS} and \eqref{II} and Theorem \ref{thm relp}.
\end{proof}

For $n=2$ formula \eqref{Rbk} takes the form:
\bean
\label{A2p}
&&
R_{\bs k}(a,b_1,b_2,c)
 = (-1)^{k_1+k_2}\,\prod_{i=1}^{k_1-k_2}
\frac{(b_1+(i-1)c)!}
{(1+a+b_1+(i+k_1-2)c-p)!}
 \\
 \notag
 &&
\times\,
\prod_{i=1}^{k_2}\
\frac{(b_2+(i-1)c)!}
{(1+b_2+(i+k_2-k_1-2)c)!}\,
\
\frac{(1+b_1+b_2+(i-2)c)!}
{(2+a+b_1+b_2+(i+k_1-3)c-p)!}
\\
&&
\notag
\times \
\prod_{i=1}^{k_1}
(a+(i-1)c)!\,
\prod_{i=1}^{k_2}
(p+(i-k_1-1)c)!
\prod_{r=1}^2\prod_{i=1}^{k_r}
\frac{(ic)!}{c!}\,.
\eean

\begin{lem}
\label{lem R-1}
Assume that $(a,b_1-1,b_2,c), (a,b_1,b_2,c) \in \mc A_{\bs k}$. Then
\bean
\label{R-1}
R_{\bs k}(a,b_1-1,b_2,c) 
&=& 
R_{\bs k}(a,b_1,b_2,c)
\prod_{i=1}^{k_1-k_2}
\frac
{1+a+b_1+(i+k_1-2)c-p}
 {b_1+(i-1)c}
 \\
 \notag
 &\times &\prod_{i=1}^{k_2}
\frac
{2+a+b_1+b_2+(i+k_1-3)c-p}
{1+b_1+b_2+(i-2)c}\,.
\eean
Assume that $(a,b_1,b_2-1,c), (a,b_1,b_2,c) \in \mc A_{\bs k}(a,b_1,b_2,c)$. Then
\bean
\label{R-2}
R_{\bs k}(a,b_1,b_2-1,c) 
&=& 
R_{\bs k}(a,b_1,b_2,c)
\prod_{i=1}^{k_2}
\frac{1+b_2+(i+k_2-k_1-2)c}
{b_2+(i-1)c}
 \\
 \notag
 &\times &\prod_{i=1}^{k_2}
\frac
{2+a+b_1+b_2+(i+k_1-3)c-p}
{1+b_1+b_2+(i-2)c}\,.
\eean
\qed

\end{lem}

By Lemmas \ref{lem S-1} and \ref{lem R-1}
the functions $S_{\bs k}(a,b_1,b_2,c)$ and $R_{\bs k}(a,b_1,b_2,c)$
defined on $\mc A_{\bs k}$ satisfy the same difference equations with respects to the shifts of variables
$b_1\to b_1-1$ and $b_2\to b_2-1$.

\begin{lem}
\label{lem verti}
Assume that $a, c$ are positive integers such that $0\,<k_1c\,\leq p-1$, 
\\
$a+(k_1-1)c \,<\,p-1$.
Then the point
\bean
\label{tr sp}
(a,b_1,b_2,c) = (a, p-1 - (a +(k_1-1)c), (k_1-k_2+1)c-1, c)
\eean
lies in $\mc A_{\bs k}$.

\end{lem}

\begin{proof} If $(a,b_1,b_2,c)$ is given by \eqref{tr sp}, 
then
\bea
&&
R_{\bs k}
 = (-1)^{k_1+k_2}\,\prod_{i=1}^{k_1-k_2}
\frac{(p-1-(a+(k_1-i)c))!}
{((i-1)c)!}
 \\
 \notag
 &&
\times\,
\prod_{i=1}^{k_2}\
\frac{((k_1-k_2+i)c-1)!}
{((i-1)c)!}\,
\
\frac{(p-1-(a+(k_2-i)c))!}
{((k_1-k_2+i-1)c)!}
\\
&&
\notag
\times \
\prod_{i=1}^{k_1}
(a+(i-1)c)!\,
\prod_{i=1}^{k_2}
(p+(i-k_1-1)c)!
\prod_{r=1}^2\prod_{i=1}^{k_r}
\frac{(ic)!}{c!}\,.
\eea
This proves the lemma.
\end{proof}

\begin{lem}
\label{lem convex}
Assume that $\tilde a, \tilde c$ are nonnegative integers such that $0<k_1\tilde c\,\leq p-1$, $\tilde a+(k_1-1)\tilde c \,\leq\,p-1$.
Denote by
$\mc A_{\bs k}(\tilde a,\tilde c)$ the set of all $(a,b_1,b_2,c)\in \mc A_{\bs k}$ such that $a=\tilde a$, $c = \tilde c$.
Then 
$\mc A_{\bs k}(\tilde a,\tilde c)$ consists of the pairs $(b_1,b_2)$
of nonnegative integers satisfying the inequalities
\bean
\label{IN1}
p-1 - (\tilde a +(k_1-1)\tilde c)\leq b_1,\qquad
 (k_1-k_2+1)\tilde c-1\leq b_2
 \eean
 and some other inequalities of the form
\bean
\label{IN2}
b_1\leq A_1,
\qquad b_2\leq A_2\qquad
b_1+b_2\leq A_{12},
\eean
where $A_1,\,A_2,\,A_{12}$ are some integers such that
\bea
&&
A_1\geq p-1 - (\tilde a +(k_1-1)\tilde c),\quad
A_2 \geq (k_1-k_2+1)\tilde c-1,
\\
&&
A_{12}\geq p -1 - (\tilde a +(k_1-1)\tilde c) + (k_1-k_2+1)\tilde c-1\,.
\eea
\end{lem}

\begin{proof}  The lemma follows from Lemmas \ref{lem apar} and \ref{lem ex p}.
\end{proof}

\begin{cor}
\label{cor 5.7}

Any point $(\tilde a, b_1,b_2,\tilde c) \in \mc A_{\bs k}(\tilde a,\tilde c)$
can be connected with the point $(\tilde a, p-1-(\tilde a+(k_1-1)\tilde c), 
(k_1-k_2+1)\tilde c-1,\tilde c)\in \mc A_{\bs k}(\tilde a,\tilde c)$
by a piece-wise linear 
path in  $\mc A_{\bs k}(\tilde a,\tilde c)$ consisting of the vectors $(0,-1,0,0)$ or $(0,0,-1,0)$.
\qed

\end{cor}

\smallskip
\noindent
{\it Proof of Theorem \ref{thm main n} for $n=2$}.
For $n=1$,\, $\bs k=(k_1)$ and the point $(\tilde a, p-1-(\tilde a+(k_1-1)\tilde c), \tilde c)$
formula \eqref{main n} holds by \cite[Theorem 4.1]{RV}. 

For $n=2$, \,$\bs k = (k_1,k_2)$ and  the point 
$(\tilde a, p-1-(\tilde a+(k_1-1)\tilde c), 
(k_1-k_2+1)\tilde c-1,\tilde c)$ 
formula \eqref{main n} holds by Lemma \ref{lem verti} and Theorem \ref{thm ind}.

For $n=2$,\, $\bs k = (k_1,k_2)$  and arbitrary
$(\tilde a, b_1,b_2,\tilde c) \in \mc A_{\bs k}(\tilde a,\tilde c)$
formula \eqref{main n} holds by Lemmas \ref{lem S-1}, \ref{lem R-1} and Corollary \ref{cor 5.7}.
Theorem \ref{thm main n} for $n=2$ is proved.
\qed

\subsection{Evaluation of $I_{0,0,0}(a,b_1,b_2,c)$}

In this section we evaluate $I_{0,0,0}(a,b_1,b_2,c)$ without using the evaluation of
$S_{\bs k}(a,b_1,b_2,c)$.

\begin{thm}
\label{thm I}
Let $\bs k=(k_1,k_2)$, $k_1>k_2> 0$ and  $a,b_1,b_2,c\in\Z_{>0}$.
Assume that
\\
 $a+(k_1-1)c<p$ and all factorials in the right-hand side of the next formula 
are factorials of the nonnegative integers less than $p$.
Then
\bean
\label{Last}
&&
I_{0,0,0}(a,b_1,b_2,c)
 = (-1)^{k_1+k_2}\,\prod_{i=1}^{k_1-k_2}
\frac{(b_1+(i-1)c)!}
{(a+b_1+(i+k_1-2)c-p)!}
 \\
 \notag
 &&
\times\,
\prod_{i=1}^{k_2}\
\frac{(b_2+(i-1)c)!}
{(b_2+(i+k_2-k_1-2)c)!}\,
\
\frac{(b_1+b_2+(i-2)c)!}
{(a+b_1+b_2+(i+k_1-3)c-p)!}
\\
&&
\notag
\times \
\prod_{i=1}^{k_1}
(a+(i-1)c-1)!\,
\prod_{i=1}^{k_2}
(p+(i-k_1-1)c-1)!
\prod_{r=1}^2\prod_{i=1}^{k_i}
\frac{(ic)!}{c!}\,.
\eean
\end{thm}

\begin{proof}
The proof is parallel to the proof of Theorem \ref{thm main n} for $n=2$.
\vsk.2>

Denote by $\mc A^I_{\bs k}$ the set of all $a,b_1,b_2,c\in\Z_{>0}$ satisfying the assumptions of Theorem \ref{thm I}.
 Notice that if $(a,b_1,b_2,c)\in \mc A^I_{\bs k}$, then
\bean
\label{ine I}
b_1\geq p-(a+(k_1-1)c), \qquad b_2\geq (k_1-k_2+1)c.
\eean

\begin{lem}
\label{lem int I} 

Formula \eqref{Last} holds if $b_2=(k_1-k_2+1)c$ and
$(a, b_1, (k_1-k_2+1)c, c)\in \mc A^I_{\bs k}$.

\end{lem}

\begin{proof}
If  $b_2=(k_{1}-k_2+1)c$, then all variables $(s_j)$
in the integrand of
 $I_{0,0,0}(a, b_1,(k_1-k_2+1)c, c)$
 are used  to reach
the monomial $\prod_{j=1}^{k_n} s_j^{k_{1}p-1}$ in the calculation of the $\F_p$-integral
$I_{0,0,0}(a, b_1, (k_1-k_2+1)c, c)$.  The remaining free variables 
$(t_i)$ all belong to the factor
\bea
\Phi_{(k_1)}(t_1,\dots,t_{k_1}, a-1, b_1,c)=
\prod_{1\leq i<i'\leq k_1}(t_i-t_{i'})^{2c}\prod_{i=1}^{k_1} t_i^{a-1}(1-t_i)^{b_1}
\eea
of the integrand and are used to calculate the coefficient of the monomial
$\prod_{j=1}^{k_1} t_i^{p-1}$.

More precisely, under the assumptions of the theorem we have
\bea
I_{0,0,0}(a, b_1, (k_1-k_2+1)c, c)  = 
  (-1)^{b_2k_2+ck_2(k_2-1)/2}\frac{(k_2c)!}{(c!)^{k_2}}
     \,S_{(k_1)}(a-1,b_1,c),
\eea
cf. the proof of Theorem \ref{thm ind}.
We have 
$  (-1)^{b_2k_2+ck_2(k_2-1)/2}=  (-1)^{ (k_{1}k_2-k_2(k_2+1)/2)c-k_2}$.

By \cite[Theorem 4.1]{RV} we have
$S_{(k_1)}(a-1,b_1,c) = R_{(k_1)}(a-1,b_1,c)$.  Hence
\bean
\label{I sp}
&&
I_{0,0,0}(a, b_1, (k_1-k_2+1)c, c)  =  (-1)^{ (k_{1}k_2-k_2(k_2+1)/2)c} (-1)^{k_1+k_2}
\\
\notag
&&
\phantom{aaaaaaaaaa}
\times \,  \frac{(k_2c)!}{(c!)^{k_2}} \
     \prod_{j=1}^{k_1} \frac{(jc)!}{c!}\,
\frac{(a+(j-1)c-1)!\,(b+(j-1)c)!}
{(a+b + (k_1+j-2)c-p)!}\,.
\eean
Denote by $R^I_{\bs k}(a,b_1,b_2,c)$ the right-hand side in \eqref{Last}.
We have
\bean
\label{RI sp}
&&
\\
&&
\notag
R^I_{\bs k}(a,b_1,(k_1-k_2+1)c,c)
 = (-1)^{k_1+k_2}\,\prod_{i=1}^{k_1-k_2}
\frac{(b_1+(i-1)c)!}
{(a+b_1+(i+k_1-2)c-p)!}
 \\
 \notag
 &&
\times\,
\prod_{i=1}^{k_2}\
\frac{((k_1-k_2+1)c+(i-1)c)!}
{((k_1-k_2+1)c+(i+k_2-k_1-2)c)!}\,
\\
 \notag
 &&
\times\,
\prod_{i=1}^{k_2}\
\frac{(b_1+(k_1-k_2+1)c+(i-2)c)!}
{(a+b_1+(k_1-k_2+1)c+(i+k_1-3)c-p)!}
\\
&&
\notag
\times \
\prod_{i=1}^{k_1}
(a+(i-1)c-1)!\,
\prod_{i=1}^{k_2}
(p+(i-k_1-1)c-1)!
\prod_{r=1}^2\prod_{i=1}^{k_i}
\frac{(ic)!}{c!}\,.
\eean
\bea
&&
=  (-1)^{k_1+k_2} (-1)^{ (k_{1}k_2-k_2(k_2+1)/2)c} 
 \,  \frac{(k_2c)!}{(c!)^{k_2}}\,
 \prod_{j=1}^{k_1} 
\frac{(ic)!}{c!}\,\frac{(a+(j-1)c-1)!\,(b+(j-1)c)!}
{(a+b + (k_1+j-2)c-p)!},
\eea
where we used the cancellation Lemma \ref{lem ca} in the last step.
Hence
\\
 $I_{0,0,0}(a, b_1, (k_1-k_2+1)c, c)  = 
R^I_{\bs k}(a, b_1, (k_1-k_2+1)c, c)$ and Lemma \ref{lem int I} is proved.  
\end{proof}

Comparing equations \eqref{B1}, \eqref{B2} and the formula for
$R^I_{\bs k}(a, b_1, b_2, c)$, we conclude that 
the functions $I_{0,0,0}(a, b_1, b_2, c)$ and $R^I_{\bs k}(a, b_1, b_2, c)$
on $\mc A^I_{\bs k}$ satisfy the same difference equations
with respect to 
the   shifts of variables
$b_1\to b_1-1$ and $b_2\to b_2-1$ and are equal if $b_2$ takes its minimal value
$(k_1-k_2+1)c$. This implies Theorem \ref{thm I}, cf. Lemmas \ref{lem verti}, \ref{lem convex}
and Corollary~\ref{cor 5.7}.
\end{proof}

\subsection{Proof of Theorem \ref{thm main n} for arbitrary $n>2$}
\label{sec sketch}

The proof is parallel to the proof of Theorem \ref{thm main n} for $n=2$.

Analogously to the proof of Theorem \ref{thm relp},
 consider the Lie algebra $\frak{sl}_{n+1}$ and its representations
$V_{\la_1}^{\F_p}$ and $ V_{\la_2}^{\F_p}$ over $\F_p$
with   highest weights
\bea
\la_1= -\frac{a}{c+p}\,\om_1 \,,
\qquad
\la_2= -\frac{b_1}{c+p}\,\om_1 -\dots - \frac{b_n}{c+p}\,\om_n \,
\eea
and highest weight vectors $v_1,\,v_2$. Consider the PBW basis $\mc B = (u)$ of the weight subspace
$V_{\la_1}^{\F_p}\ox V_{\la_2}^{\F_p}[\la_1+\la_2-\sum_{i=1}^n k_i\si_i]$ like in
the proof of Theorem \ref{thm relp}.
 We distinguish two elements of that basis:
\bea
&&
u_1 = \frac{ f_1^{k_1-k_2}[f_1,f_2]^{k_2-k_3}\dots [f_1,[f_2, \dots, [f_{n-1},f_n]\dots]]^{k_n}v_1
\ox v_2}{ (k_1-k_2)!\,(k_2-k_3)!\,\dots \,(k_n)!}\,,
\\
&&
u_2 = \frac{ f_1^{k_1}v_1 \ox f_2^{k_2}\dots f_n^{k_n}
v_2}{ (k_1)!\,(k_2)!\,\dots\,(k_n)!}\,.
\eea
For $n=2$ these vectors are the vectors $v_{0,0,0}$ and $ v_{0,k_2,0}$ in the proof of Theorem \ref{thm relp}.

\vsk.2>

To any basis vector $u\in \mc B$ we assign the weight function $W_u(t)$ defined in \cite[Section 6.1]{RSV}, here
$t$ is the collection of variables defined in  \eqref{def t}. Then we consider  the $\F_p$-integrals
\bea
I_u(a,b,c) = \int_{[\bs k]_p} \Phi(t,a,b,c) W_u(t)\,dt.
\eea 
It follows from the formulas for the weight functions that 
\bea
I_{u_2}(a,b,c) = S_{\bs k}(a-1,b_1,b_2-1,\dots,b_n-1,c),
\eea
 cf. \eqref{IS}.
It is known from the theory of KZ equations  that the vector
 \bea
 I(a,b,c) = \sum_{u\in \mc B} I_u(a,b,c)\, u
 \eea
is a singular vector in 
$V_{\la_1}^{\F_p}\ox V_{\la_2}^{\F_p}[\la_1+\la_2-\sum_{i=1}^n k_i\si_i]$.
From the singular vector condition it  follows that 
\bean
\label{IBI}
I_{u_2}(a,b,c)= B_0(a,b,c)\,I_{u_1}(a,b,c),
\eean
where $B_0(a,b,c)$ is  an explicit expression like in \eqref{II}.

Then consider the vector $I_{u_1}(a,b,c) u_1$ of the one-dimensional weight subspace
$V_{\la_1}^{\F_p}[\la_1+\la_2-\sum_{i=1}^n k_i\si_i]$. That vector satisfies the dynamical equations
defined in \cite{TV3}. The dynamical equations take the form
\bean
\label{dyne n}
I_{u_1}(a,b_1,\dots,b_i-1,\dots,b_n,c) = B_i(a,b,c) I_{u_1}(a,b,c), \qquad i=1,\dots,n,
\eean
where $B_i(a,b,c)$ are explicit products like in 
\eqref{B1} and \eqref{B2}.

Equation \eqref{IBI} and difference equations \eqref{dyne n} imply that the two functions
$S_{\bs k}(a,b,c)$ and $R_{\bs k}(a,b,c)$, defined on $\mc A_{\bs k}$,
satisfy the same difference equations with respect to the shift
of variables $b_i\to b_i-1$ for $i=1,\dots,n$.
By Corollary \ref{cor ind} we also know that the two functions are equal at the
distinguished  point
\bea
(a,p-1-(a+(k_1-1)c), (k_1-k_2+1)c-1,\dots,(k_{n-1}-k_n+1)c-1,c)\ \in\ \mc A_{\bs k}\,.
\eea
This implies that the two functions are equal (cf. Corollary \ref{cor 5.7})
and Theorem \ref{thm main n} is proved for any $n$.
The details of the proof will be published elsewhere.

\end{document}